\documentclass[leqno,12pt]{article}
\usepackage{amsmath}
\usepackage{amscd}
\usepackage{amsopn}
\usepackage{amsthm}
\usepackage{amsfonts,amssymb}
\usepackage{amsfonts,bbm}
\usepackage{latexsym}

\makeatletter
\@addtoreset{equation}{section}
\makeatother

 \newtheorem{lemma}{LEMMA}[section]
 \newtheorem{proposition}[lemma]{PROPOSITION}
 \newtheorem{corollary}[lemma]{COROLLARY}
 \newtheorem{theorem}[lemma]{THEOREM}
 \newtheorem{remark}[lemma]{REMARK}
\newcommand{\real}{\mathbbm{R}}

\newcommand{\nat}{\mathbbm{N}}

\renewcommand{\a}{\alpha}
\renewcommand{\b}{\beta}

\newcommand{\g}{\gamma}
\newcommand{\vp}{\varphi}
\newcommand{\ve}{\varepsilon}

\newcommand{\lam}{\lambda}

\newcommand{\reald}{{\real^d}}
\newcommand{\es}{E_{\mathbbm P}}

\newcommand{\on}{\quad\text{ on }}
\newcommand{\und}{\quad\mbox{ and }\quad}
\newcommand{\inv}{^{-1}}
\newcommand{\ov}{\overline}

\newcommand{\C}{\mathcal C}

\newcommand{\B}{\mathcal B}

\newcommand{\uc}{{U^c}}

\title{Unavoidable collections of balls for processes\\
        with isotropic unimodal Green function} 
\author{Wolfhard Hansen\\  
{\small Department of Mathematics, University of Bielefeld}\\
{\small 33501 Bielefeld, Germany}\\
{\small hansen$@$math.uni-bielefeld.de}}

\date{}

\begin{document}

\maketitle

\begin{abstract}
Let us suppose that we have a~right continuous Markov semigroup on~$\mathbbm R^d$, $d\ge 1$,
such that its potential kernel is given by convolution with a~function \hbox{$G_0=g(|\cdot|)$},
where $g$ is decreasing,  has a~mild lower decay property at zero, and a~very weak decay property
at infinity. 
This captures not only the Brownian semigroup (classical potential theory) 
and  isotropic $\alpha$-stable semigroups (Riesz potentials), but also more general isotropic  
L\'evy processes, where the characteristic function has a~certain lower scaling property,
and various geometric stable processes.

There always exists a~corresponding Hunt process. A~subset~$A$
of~$\reald$ is called unavoidable, if the process hits~$A$ with probability~$1$, wherever it starts.
It is known that, for any locally finite union of pairwise disjoint balls $B(z,r_z)$, $z\in Z$, which is unavoidable,
$\sum_{z\in Z} g(|z|)/g(r_z)=\infty$. The converse is proven assuming, in~addition, that,
for some $\varepsilon>0$, $|z-z'|\ge \varepsilon |z| (g(|z|)/g(r_z))^{1/d}$, whenever $z,z'\in Z$, $z\ne z'$.
It also holds, if the balls are regularly located, that is,  if their centers keep some minimal mutual  distance,
each ball of a~certain size intersects $Z$, and $r_z=g(\phi(|z|))$, where $\phi$
is a~decreasing function.

The  results generalize and, exploiting a~zero-one law,  simplify recent work by A.\,Mimica and  Z.\,Vondra\v cek.
\end{abstract}

\section{Introduction and main results}

Let $\mathbbm P=(P_t)_{t>0} $ be a~right continuous Markov semigroup on $\reald$, $d\ge 1$, 
such that its potential kernel $V_0:=\int_0^\infty P_t\,dt$ is given by convolution with a~$\mathbbm P$-excessive function 
\[
              G_0=g(|\cdot|),
\]
where $g$ is  a~decreasing   function on $\left[0,\infty\right)$ such that  
$0<g<\infty$ on $(0,\infty)$,  $\lim_{r\to 0} g(r)=g(0)=\infty$ and the following holds:
\begin{itemize} 
\item[\rm (LD)] \emph{Lower decay property}:
 There are $R_0\ge 0$ and $C_G\ge 1$ such that
\begin{equation}\label{g-decay}
       d     \int_0^r  s^{d-1}g(s) \,ds \le C_G \, r^d g(r), \qquad\mbox{  for all~$r>R_0$}.
\end{equation} 
\item[\rm (UD) ] \emph{Upper decay property at infinity}: 
There are $R_1\ge 0$, $\eta\in (0,1) $, and $K>1$ such that
\begin{equation}\label{ucd}
          g(Kr)\le \eta g(r), \qquad \mbox{  for all $r> R_1$}.
\end{equation} 
\end{itemize} 

\begin{remark}\label{crucial} {\rm
   The inequality (\ref{g-decay}) has a~very intuitive meaning: If $B$ is a~ball with radius $r$ and center $0$, 
and $\lambda_B$ denotes normed Lebesgue measure on~$B$, then the potential 
$G\lambda_B:=G_0\ast \lambda_B$ of $\lambda_B$ satisfies  $G\lambda_B(0)  \le C_G g(r)$
(and hence $G\lambda_B \le   C_G g(r)$ on $\reald$), where $g(r)$ is the value of $G_0$ at the boundary of $B$.

In Section \ref{exa} 
we shall see the following.

1. For the Brownian semigroup (classical potential theory)
and isotropic $\a$-stable semigroups (Riesz potentials)  we have $g(r)=r^{\a-d}$,   $\a\in (0,2]$, $\a<d$,
and our assumptions are satisfied with $R_0=R_1=0$. This holds as well for the more general
isotropic unimodal L\'evy semigroups considered in~\cite{mimica-vondracek}.

2.  If (LD) is satisfied for some $R_0>0$, then, for \emph{every} $R>0$, there exists 
$C_G\ge 1$ such that (\ref{g-decay}) holds for all $r>R$ (and hence the
restriction $r_z>R_0$, for all $z\in Z$, imposed below,  reduces to the requirement 
that $\inf_{z\in Z} r_z>0$). 
Analogously for (UD).

3.  If $\int_0^1 s^{d-1} g(s)\, ds<\infty$ and $g(r)\approx r^{\b-d}$, $0<\b<d$, 
 as~$r\to \infty$,  then (LD) and (UD) hold with arbitrary $R_0, R_1\in (0,\infty)$.

Subordinate Brownian semigroups   with  subordinators
having Laplace exponents of the form 
\[
\phi(\lambda)=
  \ln^\delta(1+\lambda^{\a/2}), \qquad\mbox{ $ 0<\delta\le 1$,  $\a\in (0,2]$, $\a<d$}, 
\]
provide examples (symmetric geometric stable processes, if $\delta=1$), 
where (LD) does not hold with $R_0=0$.  Here $g$ satisfies
\[
       g(r)\approx r^{-d}\ln^{-(1+\delta)} (1/r) \mbox{ as } r\to 0 \und
       g(r)\approx r^{\delta\a-d}         \mbox{ as } r\to \infty.
\]
}
\end{remark}

Let $\mathfrak X=(\Omega,\mathfrak M, \mathfrak M_t, X_t, \theta_t, P^x)$
be an associated  Hunt process on $\reald$ (for its existence see Remark \ref{ex-process},1). 
A~Borel measurable set $A$ in $\reald$  is called \emph{unavoidable},
if
\[
       P^x[T_A<\infty]=1 \qquad\mbox{ for every } x\in\reald,
\]
where  $T_A(\omega):=\inf \{t\ge 0\colon X_t(\omega)\in A\}$.
Otherwise, it is called \emph{avoidable}, that is, $A$~is avoidable, if there exists $x\in\reald$
such that $P^x[T_A<\infty]<1$. 

For all $x\in \reald$ and $r>0$, let $B(x,r)$ denote
the open ball with center $x$ and radius $r$, and let $\ov B(x,r)$ be
its closure. %the closed ball, $\ov B(x,r):=\ov {B(x,r)}$. 
Let us introduce two properties for families of balls which,  in the the classical case, have
already  been considered in~\cite{carroll,gardiner-ghergu} %and \cite{carroll}                                                       
(and where it does not make a~real difference, if we look at open or closed balls, since a~union of open balls
is unavoidable if~and only if the union of the corresponding closed balls is unavoidable;
see Remark~\ref{ex-process},2). 

Let $Z$ be a~countable set in $\reald\setminus \{0\}$ and $r_z>R_0$, $z\in Z$,
such that the balls $B(z,r_z)$ are pairwise disjoint.
We say that the balls $B(z,r_z)$, $z\in Z$, satisfy the \emph{separation condition}, if
$Z$ is locally finite and 
\begin{equation}\label{separation}
       \inf\nolimits_{z,z'\in Z,\, z\ne z'}   \frac  {|z-z'|^d}{|z|^d}        \, \frac {g(r_z)} {g(|z|)} >0.
\end{equation} 
We say that they  are \emph{regularly located}, if the following holds:
\begin{itemize}
 \item[\rm (a)] There exists $\ve>0$ such that $|z-z'|\ge \ve$, for all $z,z'\in Z$, $z\ne z'$.
 \item[\rm (b)] There exists  $R>0$ such that $B(x,R)\cap Z\ne \emptyset$, for every $x\in \reald$.
 \item[\rm (c)] There exists a~decreasing function $\phi\colon [0,\infty)\to (0,\infty)$ such that $r_z=\phi(|z|)$.
 \end{itemize} 
% \begin{alphlist}[(a)]
% \item
% There exists $\ve>0$ such that $|z-z'|\ge \ve$, for all $z,z'\in Z$, $z\ne z'$.
% \item
% There exists  $R>0$ such that $B(x,R)\cap Z\ne \emptyset$, for every $x\in \reald$.
% \item
% There is a~decreasing function $\phi\colon [0,\infty)\to (0,\infty)$ such that $r_z=\phi(|z|)$.
% \end{alphlist}

 Our main results are the following 
(where we might bear in our mind that $1/g(r)$ is approximately the capacity of balls having radius $r$,
that is, the total mass of their equilibrium measure; see Proposition \ref{cap}).

\begin{theorem}\label{main} 
If the balls $B(z,r_z)$, $z\in Z$, satisfy the separation condition, 
then their  union $A$  is unavoidable provided
\[
\sum\nolimits_{z\in Z} \frac{g(|z|)} {g(r_z)}=\infty.
\]
\end{theorem}

\begin{corollary}\label{main-corollary}
Suppose that the balls $B(z,r_z)$, $z\in Z$,  are regularly located.
Then their union $A$ is unavoidable if and only if 
\[
\int_1^\infty \frac {r^{d-1} g(r)}{g(\phi(r))}\,dr =\infty.
\]
\end{corollary} 

The converse in Theorem \ref{main} is already known without any restriction on the balls 
and assuming only $\lim_{r\to\infty} g(r)=0$ instead of (UD) (see  \cite[Theorem 6.8]{HN-unavoidable};
the inequality $R_1^{\ov B(z,r_z)}\le g(|z|)/g(r_z)$, which is used in its proof, holds trivially, since $g$ is decreasing).

\begin{proposition}\label{known}
Let $A$ be an unavoidable union of  
balls  $B(z,r_z)$, $z\in Z$. %, $Z \subset \reald$. 
Then $\sum_{z\in Z} g(|z|)/g(r_z)=\infty$ and $\sum_{z\in Z} 1/g(r_z)=\infty$.
\end{proposition} 

\begin{remark}\label{remark-main} {\rm
1.  In the classical case, Theorem \ref{main}  is \cite[Theorem 6]{gardiner-ghergu} (for unavoidableness  under
a~ weaker separation property see \cite{carroll}) and Corollary~\ref{main-corollary} is  \cite[Theorem~2]{carroll}.

2.  In the more general  case of isotropic unimodal  L\'evy processes, where the characteristic function satisfies
a~lower scaling condition (and (LD), (UD) hold with  $R_0=R_1=0$), both Theorem \ref{main}, its converse, 
and Corollary~\ref{main-corollary} are proven
in~{\cite{mimica-vondracek}}. We shall use the same method of considering finitely many countable 
unions of concentric shells,
but have to overcome additional difficulties caused by having only a~rather weak estimate
for  the exit distribution of balls (compare \cite[Lemma~2.2]{mimica-vondracek},
 going back to \cite[Corollary~2]{grzywny}, and
Proposition~\ref{hame}).
Nevertheless our proof for Theorem~\ref{main} can be simpler, since starting with an avoidable union $A$
and an arbitrary $\delta>0$, we may assume without loss of generality
that $P^0[T_A<\infty]< \delta$ (using  Proposition~\ref{so-simple}   and translation invariance).

3.  If the balls $B(z,r_z)$, $z\in Z$, are regularly located, then  
\[
    \sum\nolimits_{z\in Z} \frac{g(|z|)} {g(r_z)}=\infty \quad
    \mbox{ if and only if } \quad \int_1^\infty \frac {r^{d-1} g(r)}{g(\phi(r))}\,dr =\infty.
\]
This is fairly obvious (see \cite[Lemma~4.1]{mimica-vondracek}) and allows us
to reduce Corollary \ref{main-corollary} to a~consequence of Theorem \ref{main} by first 
treating a~simple case (see Proposition \ref{MV-L4.2}).   
}
\end{remark}

In view of the second statement in Proposition \ref{known} let us mention the following part
 of \cite[Theorem 6.8]{HN-unavoidable} (where only $\lim_{r\to \infty} g(r)=0$ 
instead of~(UD) is needed).
See also~\cite{HN-champagne} 
for the result in classical potential theory.

\begin{theorem}\label{iso} Suppose that {\rm(LD)} holds with $R_0=0$. 
% \footnote{In fact, it is sufficient
% to know that, for some $r_0>0$, (\ref{g-decay}) holds for every $0<r<r_0$, and that  $\lim_{r\to \infty} g(r)=0$
% holds instead of (UD).}
Let  $h\colon (0,1)\to (0,1)$ with   $\lim_{t\to 0} h(t)=0$,  let  $\vp\in \C(\reald)$, $\vp >0$,
and $\delta>0$. 
Then  there exist a~locally finite set $Z$ in $\reald$ and $0<r_z<\vp(z)$,
 $z\in Z$,  
such that the  balls $\ov B(z,r_z)$ are pairwise disjoint,
the union  of all~$\ov B(z,r_z)$ is unavoidable, and 
\[
\sum\nolimits_{z\in Z}  h(r_z)/g(r_z)<\delta.
\]
\end{theorem} 

In Section 2, we shall first take a~closer look at the properties (LD) and (UD) and
then show that our assumptions cover the isotropic unimodal processes considered
in~\cite{mimica-vondracek} and geometric stable processes.
% look at cases such as  $g(r)\approx r^{-d}\ln^{-2} (1/r)$
% as $r\to 0$ and $g(r)\approx r^{\a-d} $ as $r\to \infty$.
In Section~3, we shall discuss     
some general potential theory of the semigroup~$\mathbbm P$, where, as in \cite{HN-unavoidable},        
at the beginning (LD) and (UD) are  replaced by the   weaker properties $\int_0^1 r^{d-1}g(r)\, dr<\infty$
and  $\lim_{r\to \infty}g(r)=0$.  In Section~4, we prove Theorem~\ref{main},
and the proof of Corollary~\ref{main-corollary} is given in Section~5.

\section{Examples}\label{exa}

Let us first consider an arbitrary  positive decreasing function~$g$ on~$(0,\infty)$ and 
 write down a~few elementary facts justifying, in particular, our statements in Remark~\ref{crucial}. 

Given $R_0\ge 0 $, we say that (LD) holds on $(R_0,\infty)$,  if there  exists  $C\ge 1$ such that  
\begin{equation}\label{gCr}
d      \int_0^r s^{d-1} g(s)\,ds \le C r^d g(r), \qquad \mbox{ for every } r>R_0.
\end{equation} 
Similarly, given $0\le R_1<\infty$, we say that  (UD) holds on $(R_1,\infty)$, if there exist $K>1$    
and $\eta\in (0,1)$ such that 
\begin{equation}\label{gMr}
g(Kr)\le \eta g(r), \qquad\mbox{  for every }r>R_1.
\end{equation} 
 
\begin{lemma}\label{LD} 
\begin{itemize} 
\item[\rm 1.]
If there is  a~function $\vp>0$  on $(0,1)$ with    $\int_0^1 \g^{d-1}\vp(\g) \, d\g<\infty$ and
\[
  g(\g r)\le \vp(\g) g(r), \qquad\mbox{  for all }\g\in (0,1) \mbox{  and } r>0,
\]
then {\rm (LD)} holds on $(0,\infty) $. 
\item[\rm 2.]
Let $f(r):=r^d g(r)$, $R_0\ge 0$,    $\kappa, C\in (0,\infty)$. If  $\int_0^1 s\inv f(s)\,ds<\infty$,   
$f\ge \kappa$ on $(R_0,\infty)$, and 
\[
\int_{R_0}^r s\inv f(s)\,ds \le C f(r),\qquad\mbox{ for every $r>R_0$},
\]
then {\rm(LD)} holds on $(R_0,\infty)$. 
\item[\rm 3.]
If $0<R<R_0$ and   {\rm(LD)} holds on $(R_0,\infty)$,  then {\rm(LD)} holds on~$( R,\infty)$.
\end{itemize} 
\end{lemma}

\begin{proof} 
1.  For every $r>0$, 
\[
\int_0^r s^{d-1} g(s)\,ds = r^d\int_0^1 \g^{d-1} g(\g r)\,d\g
\le   r^d  g(r) \int_0^1 \g^{d-1}  \vp(\g)\, d\g. 
 \] 

2.  Clearly, $c:=  \int_0^{R_0} s^{d-1} g(s)\,ds<\infty$. For every    $r>R_0$,
\[
    \int_0^r s^{d-1} g(s)\,ds= c  +    \int_{R_0}^r  s\inv  f(s)\,ds  \le  (c\kappa\inv +C)   f(r)=  (c\kappa\inv +C) r^d g(r).
\]

3.  
Let $0<R< R_0<R_1$ and  assume that (\ref{gCr}) holds. Defining $\tilde C:=C(  R_1/R)^d$
we obtain that,  for every $r\in [R,R_0]$, 
\[
       d \int_0^r s^{d-1} g(s) \,ds\le C   R_1^d g(  R_1) =\tilde C  R^d g(  R_1)
     \le \tilde C r^d g(r).
\]
\end{proof}

\begin{lemma}\label{UD} Let $0\le R_1<\infty$.    
\begin{itemize} 
\item[\rm 1.]
If  there is a~function $\vp>0$ on $(R,\infty)$, $R>0$,  with $\lim_{\lambda\to \infty} \vp (\lambda)=0$
and $g(\lambda r)\le \vp (\lambda) g(r)$, for all $\lambda\ge R$ and $r>R_1$,  
% If there exist   $\b,C\in (0,\infty)$  such that  
%  $g(\lambda r)\le C \lambda^{-\b} g(r)$, for all $r\in (R_1,\infty)$ and~$\lambda\ge 1$,
then  {\rm (UD)} holds on~$(R_1,\infty)$.
\item [\rm 2.] 
 If $0<R<R_1$ and   {\rm(UD)} holds on $(R_1,\infty)$,  then {\rm(UD)} holds on~$( R,\infty)$.
\item [\rm 3.]
If    {\rm (UD)} holds on $(R_1,\infty)$, then, for every   $\delta>0$,  there exists $K>1$ such that 
$g(Kr)\le \delta g(r)$ for every $r>R_1$. 
\end{itemize} 
\end{lemma} 

\begin{proof}  
1. We take $K\ge R$ such that  $\vp(K) <\eta$. 

2.  Let $c:=R_1/R$,  $r> R$. Then $cr>R_1$ and $g(Kcr)\le \eta g(cr)\le\eta g(r)$. 

3. We choose $m\in\nat$ such that $\eta^m<\delta$ and  replace $K$ by $K^m$.
\end{proof}

If $0<\a<d$ and $g(r)=r^{\a-d}$, then, by  Lemmas \ref{LD} and \ref{UD},  (LD) and (UD) hold on $(0,\infty)$. 
  So our assumptions are satisfied by Brownian motion
and isotropic $\a$-stable processes with $0<\a\le 2$, $\a<d$.

 Let us  observe  next that, more generally,   our assumptions are satisfied by  the isotropic unimodal L\'evy processes 
$\mathfrak X=(X_t, P^x)$ studied in~\cite{mimica-vondracek}, where 
the characteristic function $\psi$ for ~$\mathfrak X$ (characterized by 
$  e^{-t\psi(|x|)}=   E^0 [e^{i\langle x,X_t\rangle}]$,  $t>0$) 
is supposed to satisfy the following \emph{weak lower scaling 
condition}: There exist  $\a>0$ and   $C_L>0$   such that 
\[
\psi(\lambda r)\ge C_L\lambda^\a \psi(r),\qquad \mbox{ for all }\lambda\ge 1\mbox{ and }r>0
\]
(see   \cite[(1.4)]{mimica-vondracek} and the subsequent list of  examples in \cite{mimica-vondracek}). 
Then, by \cite[Lemma~2.1]{mimica-vondracek} (see also \cite[Proposition 1, Theorem~3]{grzywny}), 
 there exists a~constant $C\ge 1$ such that, for all $r>0$ and $0<\g\le 1$, 
\begin{eqnarray} 
\frac{C\inv} {r^{d} \psi(1/r)} \le&g(r)&\le \frac C  {r^{d} \psi(1/r)} ,\label{g-approx}\\[1.5mm]
C\inv \g^{2-d} g(r)\le&g(\g r)&\le C\g^{\a-d} g(r). \label{gg-ineq}
\end{eqnarray} 
By  Lemma \ref{LD},1 and the second inequality of (\ref{gg-ineq}), (UD) holds on $(0,\infty)$.
Replacing, in the first inequality of (\ref{gg-ineq}),  $r$ by $\lambda r$ and $\g$  by $1/\lambda$,
we see that $g(\lambda r)\le C \lambda^{2-d} g(r)$, for all $r>0$ and $\lambda\ge 1$. Hence
(UD) holds on $(0,\infty)$, by  Lemma \ref{UD},1, provided $d\ge 3$. For the case $d\le 2$, see
\cite[Section 6]{mimica-vondracek}. 

Further, since the transition kernels $P_t$ are given by convolution with positive functions
$p_t$ (see, for example, \cite{hawkes}) satisfying $p_s\ast p_t=p_{s+t}$,  $s,t>0$, we have
\[
      G_0=\int_0^\infty p_t\, dt \in \es.
\] 
%  Let us  observe  next that, more generally,   our assumptions are satisfied by  the isotropic unimodal L\'evy processes 
% $\mathfrak X=(X_t, P^x)$ studied in~\cite{mimica-vondracek} with
% \begin{equation}\label{g-psi}
% \frac{C\inv} {r^{d} \psi(1/r)} \le  g(r)\le \frac C  {r^{d} \psi(1/r)} , \qquad\mbox{ for every }r>0,
% \end{equation} 
% where $\psi(|\cdot|)$ denotes the characteristic function for ~$\mathfrak X$, that is, for $t>0$,
% $  e^{-t\psi(|x|)}=   E^0 [e^{i\langle x,X_t\rangle}]$ (see \cite[Lemma 2.1]{mimica-vondracek} and 
% \cite[Proposition 1, Theorem~3]{grzywny};
% for examples of such  L\'evy processes see \cite[p.~1305, below (1.4)]{ mimica-vondracek}).

% Indeed, since the transition kernels $P_t$ are given by convolution with positive functions
% $p_t$ (see, for example, \cite{hawkes}) satisfying $p_s\ast p_t=p_{s+t}$,  $s,t>0$, we have
% \[
%       G_0=\int_0^\infty p_t\, dt \in \es.
% \] 
% Moreover, the essential assumption in \cite{mimica-vondracek}  is
% the following weak~\emph{weak lower scaling 
% condition} for $\psi$ (see   \cite[(1.4)]{mimica-vondracek}): There exist  $\a>0$ and   $C_L>0$   such that 
% \[
% \psi(\lambda r)\ge C_L\lambda^\a \psi(r),\qquad \mbox{ for all }\lambda\ge 1\mbox{ and }r>0.
% \]
% Replacing  $\lambda$ by $1/\g$ and $r$ by $1/r$, we have 
% $\psi(1/(\g r)) \ge  C_L \g^{-\a} \psi(1/r)$, and hence, by (\ref{g-psi}),
% \[
%     g(\g r) \le C^2C_L\inv \g^{\a-d} g(r),\qquad\mbox{for all $0<\g\le 1$ and  $r>0$}.    
% \]
% By  Lemmas \ref{LD} and \ref{UD}, we obtain that
%  (UD) and (LD) hold on $(0,\infty)$.
Moreover,  the separation condition   (1.6)  in
 \cite[Theorem~1.1]{mimica-vondracek} is our separation  condition~(\ref{separation}).

Now let us look at a~subordinate Brownian semigroup, where   $\a\in (0,2]$, $\a<d$, $0<\delta\le 1$,
%$ 0<\delta\le 1$,  $\a\in (0,2]$, $\a<d$, 
and the Laplace exponent of the subordinator is 
\[
\phi(\lambda)=
  \ln^\delta(1+\lambda^{\a/2}). 
\]
If $\delta=1$, then, by  \cite[Theorem 3.2 and Remark 3.3]{sikic-song-vondracek}),
\[
       g(r)\approx r^{-d}\ln^{-2} (1/r) \mbox{ as } r\to 0 \und
       g(r)\approx r^{\a-d}         \mbox{ as } r\to \infty,
\]
and (LD) certainly does not hold with $R_0=0$, since, for $r>0$, 
\begin{equation}\label{geo-int}
\int_0^r s\inv \ln^{-2}(1/s)\,ds=\ln^{-1}(1/r). 
\end{equation} 
In the general case\footnote {The author is indebted
to T.~Grzywny for   informations in this case.}, 
 we have $\phi'(\lambda)/\phi^2(\lambda)
\approx \lambda\inv \ln^{-(1+\delta)}(\lambda)$. By  \cite[Proposition~4.5]{kim-mimica}, 
we obtain that
\[
       g(r) \approx r^{-d-2} \phi'(r^{-2})/\phi^2(r^{-2}) \approx r^{-d} \ln^{-(1+\delta)} (1/r) \quad\mbox{ as }r\to 0.     
\]
Further,  by \cite[Theorem 3.3]{rao-song-vondracek}, $g(r)\approx r^{\delta\a-d}$ as $r\to \infty$.
Thus,  by Lemmas~\ref{LD} and~\ref{UD},  our assumptions in Section 1
are satisfied taking any $R_0,R_1\in (0,\infty)$.

\section{Potential theory of  {\boldmath $\mathbbm P$}}\label{potential-theory}

For the moment, let us  assume that the right continuous semigroup $\mathbbm P$ is only sub-Markov
and that, instead of (LD) and (UD), 
\begin{equation}\label{int-infty}
\int_0^1   r^{d-1}g(r)\,dr<\infty \und \lim\nolimits_{r\to \infty} g(r)=0.
\end{equation}

Let $\B(\reald)$ ($\C(\reald)$, respectively) denote the set of all 
Borel measurable numerical functions (continuous real functions, respectively) on $\reald$.
We recall that  the potential kernel $V_0=\int_0^\infty P_t\,dt$ is
given by
\[
     V_0 f(x):=G_0\ast f (x)=\int G_0(x-y) f(y)\,dy,\qquad f\in \B^+(\reald), \, x\in \reald.
\]
Let $\es$ denote the set of all $\mathbbm P$-excessive functions, that is,
$\es$ is the set of all $v\in \B^+(X)$ such that $\sup_{t>0} P_t v=v$. We
note that $V_0(\B^+(\reald))\subset \es$. If $f\in\B^+(\reald)$ 
is bounded and has compact support, then $V_0f\in \C(\reald)$ and 
$V_0f$ vanishes at infinity, by (\ref{int-infty}). This leads to 
the following results   in  \cite[Section 6]{HN-unavoidable}
(for the definition of balayage spaces
and their connection with sub-Markov semigroups see \cite{BH}, \cite{H-course}, or 
\cite[Section 8]{HN-unavoidable}).

\begin{theorem}\label{bal-space}
$(\reald,\es)$ is a~balayage space such that every point in $\reald$ is polar
and Borel measurable finely open sets $U\ne \emptyset$ have strictly
positive Lebesgue measure.
\end{theorem} 

\begin{remark}\label{ex-process}{\rm
1. There exists a~ Hunt process $\mathfrak X=(\Omega,\mathfrak M, \mathfrak M_t, X_t, \theta_t, P^x)$
 on~$\reald$ with transition  semigroup $\mathbbm P$ (see \cite[IV.7.6]{BH}). 

2. Every open ball $B(x,r)$, $x\in\reald$, $r>0$, is finely dense
in the closed ball~$\ov B(x,r)$ (see \cite[Proposition 6.4]{HN-unavoidable}; the fine topology is the coarsest topology
such that every function in~$\es$ is continuous).
}
\end{remark} 

For every subset $A$ of $\reald$, we have a~reduced function $R_1^A$:
\[
          R_1^A:=\inf\{ v\in \es\colon v\ge 1\mbox{ on } A\}.
\]
Obviously, $R_1^A\le 1$, since $1\in \es$. Hence  $R_1^A=1$ on $A$.
If $A$ is  open, then $R_1^A\in\es$. For a~general subset $A$,
the greatest lower semicontinuous minorant~$\hat R_1^A$  of $R_1^A$
(which is also the greatest finely lower semicontinuous minorant of $R_1^A$) is contained 
in~$\es$. It~is known that $\hat R_1^A=R_1^A$ on $A^c$. 
%For  Borel measurable $A$ and $x\in\reald $,
If~$A$~is  Borel measurable, then, for every $x\in\reald $,
\begin{equation}\label{connection}
  {              R_1^A(x)= P^x[T_A<\infty]}
\end{equation} 
 (see \cite[VI.3.14]{BH}). 
The zero-one law  (\ref{0-1})  will be the key to our proofs of both 
Theorem~\ref{main} and Corollary~\ref{main-corollary}.

\begin{proposition}\label{so-simple}
Suppose that   $\mathbbm P$ is a~Markov semigroup.
Then the  constant function $1$ is harmonic and, for each set $A$ in $\reald$, 
\begin{equation}\label{0-1}
                  R_1^A=1 \quad \mbox{ or }\quad  \inf\nolimits_{x\in\reald} R_1^A(x)=0.
\end{equation} 
\end{proposition} 

\begin{proof} Having $P_t1=1$, for every $t>0$, we know, by \cite[Theorem 6.3,4]{HN-unavoidable}), 
that $1$ is harmonic. Moreover, by \cite[Proposition 2.3,1]{HN-unavoidable}), (\ref{0-1}) holds.
\end{proof} 

To illustrate that (\ref{0-1}) is  almost trivial, let us suppose that $R_1^A\in \es$
(which is true in our applications) and let $\g:=\inf_{x\in\reald} R_1^A(x)$.  
 Since $\es$~is a~cone, we  trivially  have $R_{1-\g}^A=(1-\g)R_1^A\in \es$. 
So $v:=\g +R_{1-\g}^A\in \es$ and $v= 1$ on~$A$, hence $v\ge R_1^A$. Moreover, $w:=R_1^A-\g\in \es$ and $w= 1-\g$ on $A$,
hence $w\ge R_{1-\g}^A$. Therefore $R_1^A=v=\g+R_{1-\g}^A=\g+(1-\g) R_1^A$.
Thus $\g R_1^A=\g$, that is, $\g=0$ or $R_1^A=1$.

For all $x,y\in \reald$, let    
\[
    G_y(x):=  G(x,y):= G_0(x-y),
\]
and let us recall that, by definition, a~potential is a~positive superharmonic 
function with greatest harmonic minorant~$0$. The next result is 
essentially \cite[Theorem 6.6]{HN-unavoidable}.

\begin{theorem}\label{real-green}
\begin{itemize}
\item[\rm 1.]
The function $G$ is symmetric and   continuous. 
\item[\rm 2.]
For every $y\in\reald$, $G_y$ is a~potential  with superharmonic support~$\{y\}$.
\item[\rm 3.]
If $\mu$ is a~measure on $\reald$ with compact support, then 
$
G\mu:=\int G_y\,d\mu(y)
$
is a~potential,  and the support of~$\mu$ is the superharmonic support of~$G\mu$. 
\item[\rm 4.]
For every potential $p$ on $\reald$, there exists a~{\rm(}unique{\rm)} measure $\mu$ on~$\reald$
such that $p=G\mu$. 
\end{itemize} 
\end{theorem}

For every ball $B$ let $|B|$ denote the Lebesgue measure of $B$   
 and let~$\lambda_B$ denote normalized Lebesgue measure on $B$
(the measure on $B$ having density $1/|B|$ with respect to Lebesgue measure).

Let us now fix $R_0\ge 0$ and assume that (\ref{g-decay})  holds, that is, for $r>R_0$,
\begin{equation}\label{vg}
       G\lambda_{B(0,r)}(0) = \frac 1{ |B(0,r)|}        \int_{B(0,r)} G_y(0)\, dy \le C _G \,g(r).
\end{equation} 
Then, in fact (see \cite[(6.9)]{HN-unavoidable}), 
\begin{equation}\label{vg-global}
             G\lam_{B(0,r)} \le C_G g(r), \quad \mbox{ for every }r>R_0.
\end{equation} 

Moreover,  since  $g(r/2)\le g$ on $(0,r/2)$ and $d\int_0^{r/2}   s^{d-1}\,ds=(r/2)^d$,    
we see that there exists  $1\le C_D\le 2^d C_G$ such that, for all $r>R_0$,    
\begin{equation}\label{g-doubling}
       g(r/2) \le  C_D g(r)  \qquad \mbox{ (doubling property)}. 
\end{equation}
To simplify our estimates, let us define, once and for all,
     \[
     {  c:=\max\{C_D,C_G\}.}
\]

If $B$ is an open ball, then $R_1^B=R_1^{\ov B}$ is a~continuous potential and, 
by Theorem \ref{real-green}, there exists a~unique measure $\mu$ on $\ov B$,
the \emph{equilibrium measure} for~$B$, such that  $R_1^B=G\mu$.
For  measures $\nu$ on $\reald$, let $\|\nu\|$ denote their total mass.
The following holds (cf.\ \cite[Proposition~6.7]{HN-unavoidable} in the case $r_0=\infty$; 
assuming $r>R_0$ its simple proof carries over word by word).

\begin{proposition}\label{cap}
Let $r>R_0$, $B:=B(0,r)$, and let $\mu$ be the equilibrium measure for $B$. Then
\[
\frac {g(|\cdot|)}{g(r)}\ge  R_1^B\ge c\inv \,  \frac{g(|\cdot|+r)}{g(r)}
\und 
 c\inv\, \frac 1{g(r)}\le \|\mu\|\le c\,\frac1{g(r)}.
\]
\end{proposition} 

The following  well known fact will be used in the proofs of   Proposition~\ref{hame}
and Lemma~\ref{shell-1}. By (\ref{connection}),  it is an immediate 
 consequence of the strong Markov property.
For the convenience of the reader we write down its short proof
(a~corresponding argument based on iterated balayage  can be given using \cite[VI.2.9]{BH}).

\begin{lemma}\label{subset-U}
Let  $A$ be a~Borel measurable set in an open set $U\subset \reald$ and $\g>0$ such that  $R_1^A\le \g$ on $U^c$.
Then  $P^x[T_A<T_\uc]\ge R_1^A(x)- \g$, for every $x\in U$.
\end{lemma} 

\begin{proof}  Let $\tau:=T_\uc$ and $x\in U$. We obviously have the identity
\[
[T_A<\infty]\setminus [T_A<\tau] =[\tau\le T_A<\infty] = [\tau\le T_A]\cap \theta_\tau\inv ( [ T_A<\infty]).
\]
Since $X_\tau\in \uc$ on  $[\tau<\infty]$,   the strong Markov property yields that
\[
  P^x([\tau<T_A]\cap \theta_\tau\inv  [ T_A<\infty])=\int_{[\tau <T_A]} P^{X_\tau} [T_A<\infty]\,dP^x\le \g,
\]
and hence $P^x[T_A<\infty]-P^x[T_A<\tau]\le \g$.
\end{proof} 

For every $r>0$, we introduce the (closed) shell
\[
           S(r):=\ov B(0,3r)\setminus B(0,r).
\]
The following estimate of the probability for hitting a~shell $S(r)$ before leaving a~ball
 $B(0,Mr)$, $M$ large, will be sufficient for us (see 
  \cite[Lemma~2.2]{mimica-vondracek}, going back to~{\cite[Corollary~2]{grzywny},
for a~much stronger estimate which is used \cite{mimica-vondracek}). 

\begin{proposition}\label{hame}    
Let $r>R_0$, $\eta:=c^{-3}/2$,    $M > 3$, and $g((M-2) r)\le \eta g(r)$.  
Then
\begin{equation}\label{hame-loc}
P^0[T_{S(r)} < T_{B(0,Mr)^c}]  \ge \eta. 
\end{equation} 
\end{proposition} 

\begin{proof}
We choose $z\in \partial B(0,2r)$ and take $B:=B(z,r)$. Then  $B$ is contained in~$S(r)$.
By Proposition~\ref{cap},  
\begin{equation}\label{r-lower}
              R_1^B(0)\ge c\inv\,\frac{g(|z|+r)}{g(r)}=  c\inv\, \frac{g(3r)}{g(r)}\ge c^{-3}=2\eta,
\end{equation} 
whereas, for every $y\in B(0,Mr)^c$,                
\begin{equation}\label{r-upper}
             R_1^B(y)\le \frac{g(|y-z|)}{g(r)}\le \frac{g((M-2)r)}{g(r)}\le \eta.
\end{equation} 
The proof is finished by Lemma \ref{subset-U}. 
\end{proof} 

The next simple result on  comparison of potentials will be sufficient for us
(see the proof of \cite[Theorem 5.3]{HN-unavoidable} for a~much more
delicate version; cf.\ also the proof of \cite[Theorem 3]{aikawa-borichev}).

\begin{lemma}\label{comparison}
Let $Z\subset \reald$ be finite and $r_z>R_0$, $z\in Z$,  such that, for $z\ne z'$,
 $B(z,r_z)\cap B(z',3r_{z'})=\emptyset$.  Let $w\in  \es$ and, for every $z\in Z$, let 
$\mu_z,\nu_z$ be measures on~$\ov B(z,r_z)$ such that   $G\mu_z\in \C(\reald)$, $G\mu_z\le w$, and $\|\mu_z\|\le \|\nu_z\|$.
Then  $\mu:=\sum_{z\in Z} \mu_z$ and $\nu:=\sum_{z\in Z}\nu_z$ satisfy 
 \begin{equation}\label{comp-mu-nu}
             G\mu\le  w+  c G\nu.
 \end{equation} 
\end{lemma} 

\begin{proof} 
Let $z,z'\in Z$, $z'\ne z$, and $x\in \ov B(z,r_z)$. For all $y,y'\in \ov B(z',r_{z'})$,  $|y-y'|\le 2r_{z'}\le  |x-y'|$,
 hence $R_0<r_{z'}<|x-y|\le  2|x-y'|$ and  $g(|x-y'|)\le c g(|x-y|)$. By integration, 
$G\mu_{z'} (x)\le c G\nu_{z'}(x)$. Therefore
\begin{equation}\label{comparison-munu}
        G\mu(x)= G\mu_z(x)+\sum\nolimits_{z'\in Z, z'\ne z} G\mu_{z'}(x)\le w(x)+c G\nu(x).
\end{equation} 
Thus $G\mu\le w+c G\nu$ on the union of the balls $\ov B(z,r_z)$, $z\in Z$. By the minimum principle
 \cite[III.6.6]{BH}, the proof is finished.
\end{proof} 

\begin{remark} {\rm If each $G\mu_z$ is only bounded by some potential in $\C(\reald)$, but
there exists $\g>1$ such that  $B(z,\g r_z)\cap B(z', 3 r_{z'})=\emptyset$, 
whenever $z\ne z'$, then (\ref{comparison-munu}) holds for all $x\in B(z,\g r_z)$, $z\in Z$,  and (\ref{comp-mu-nu})
follows as well.
 }
\end{remark}

% \begin{remark} {\rm 
% For every  $\g>1$, there exists $k\in\nat$  such that assuming 
% $z\ne z'$, and  $G\mu_z\le p_z$ for some continuous real potential $p_z$, we obtain $G\mu\le w+c^k G\nu$. 
%  }
% \end{remark}

\section{Proof of Theorem \ref{main}} 

From now on let us suppose that the assumptions introduced at the beginning of  Section 1 are satisfied.
We recall that in many cases (LD) and (UD) hold with $R_0=0$ and $R_1=0$.
If not, we may assume without loss of generality that  $R_0$ and~$R_1$, respectively,
while being strictly positive, are as small as we want.
 
We prepare the proof of Theorem \ref{main} by a~first application of Lemma \ref{comparison}.

\begin{lemma}\label{shell-1}
Let $\rho>\max\{R_0,R_1\}$,  $0<\ve\le 1/4 $. Let  $Z$ be a~finite subset of~$S(\rho)$  and $R_0<r_z\le |z|/4$, $z\in Z$,  such that 
the balls $B(z,4r_z)$ are pairwise disjoint and
\[
 |z-z'|  \ge 4\ve |z| \bigl({g(|z|)}/{g(r_z)}\bigr)^{1/d},\qquad\mbox{ whenever }   z\ne z'.
\]
Let
$           C:=  1+ (4/\ve)^dc^3$, $\delta:=(2C  c^4)\inv$,
$M>4$, and suppose  $g((M-3)\rho)\le \delta g(\rho)$.

 Then the union $A$   of the balls $B(z,r_z)$, $z\in Z$, satisfies
\[
P^x[T_A<T_{B(0,M\rho)^c}]   \ge \delta\sum\nolimits_{z\in Z} g(|z|)/g(r_z), \qquad\mbox{for every $x\in B(0,3\rho)$}.
\]
\end{lemma}

\begin{proof}  Let $B:=B(0,4\rho)$.  By (\ref{vg-global}), 
\begin{equation}\label{gl-est}
    G\lambda_B\le c  g(4  \rho).
\end{equation} 
For  $z\in Z$, let
\[
     \tilde r_z:=\max\bigl\{r_z, \ve |z| \bigl({g(|z|)} /{g(r_z)}\bigr)^{1/d}\bigr\} 
\]
so that $ B(z,\tilde r_z)\cap B(z',3\tilde r_{z'})=\emptyset$, whenever  $ z\ne z'$.

For the moment,  fix $z\in Z$.  
Since $\max\{r_z,\ve |z|\}\le  |z|/4<\rho$ and \hbox{$g(|z|)/g(r_z)\le 1$}, we know that $B(z,\tilde r_z)\subset B$.
 Moreover, 
\begin{equation}\label{est-tilder} 
\tilde r_z^{-d}\le \ve^{-d}   \,\frac{g(r_z)}{|z|^dg(|z|)}\le \ve^{-d} \,\frac{g(r_z)}{\rho^d g( 4\rho)}.
\end{equation} 
 Let  $\mu_z$ be the equilibrium measure for $B(z,r_z)$,
that is, $G\mu_z=R_1^{B(z,r_z)}$. Then $\|\mu_z\|\le c g(r_z)\inv$, by Proposition \ref{cap}.
We define 
\begin{equation}\label{def-nuz}
                       \nu_z:= \|\mu_z\| \lambda_{B(z,\tilde r_z)}=\b_z 1_{B(z,\tilde r_z)} \lambda_B,
\end{equation} 
where,  by (\ref{est-tilder}),
\begin{equation}\label{est-density}
       \b_z= \|\mu_z\|\, \frac {|B|} {|B(z,\tilde r_z)|} \le  cg(r_z)\inv (4\rho/\tilde r_z)^d  \le (4/\ve)^{d} c g(4\rho)\inv=:\b.
\end{equation} 
Let $\nu:=\sum_{z\in Z} \nu_z$. Since the balls $B(z,\tilde r_z)$, $z\in Z$, 
are pairwise disjoint subsets of~$B$, we conclude,  by
(\ref{def-nuz}),  (\ref{est-density}), and  (\ref{gl-est}),  that
\begin{equation}\label{gnu-est}
                        G\nu \le \b G\lambda_B\le (4/\ve)^d c^2.
\end{equation}
Next let $\mu:=\sum_{z\in Z} \mu_z$ so that 
\[
p:= \sum\nolimits_{z\in Z} R_1^{B(z,r_z)} = G\mu.
\]
By  Lemma \ref{comparison},   $ G\mu\le 1+c G\nu$, and hence  $p\le C$, by (\ref{gnu-est}) and our definition of $C$. 
Therefore, by the minimum principle \cite[III.6.6]{BH},  we obtain that $              C\inv p\le R_1^{\ov A}= R_1^{A}$. 
Trivially,   $R_1^{A}\le p$. Thus
\begin{equation}\label{p-bound} 
              C\inv p\le R_1^{A}\le p.
\end{equation} 

Let $U:=B(0,M\rho)$ and $z\in Z$. By Proposition \ref{cap},  for   $y\in U^c$, 
\[
      g(r_z)    R_1^{B(z,r_z)} (y) \le g(|y-z|)\le g((M-3)\rho)\le \delta g(\rho), 
\]
whereas, for every $x\in B(0,3\rho)$, 
\[
   g(r_z)  R_1^{B(z,r_z)} (x) \ge c\inv  g(|x-z|+\rho)\ge  c\inv g(7\rho)      \ge   c^{-4}  g(\rho).
\]
Defining $\g:=\sum\nolimits_{z\in Z} g(\rho)/g(r_z)$ we hence see,  by (\ref{p-bound}), that
\[
                R_1^{A} \le  \delta \g \on U^c, \qquad \quad R_1^{A} \ge 2\delta \g \on B(0,3\rho).
\]
By Lemma \ref{subset-U}, for every $x\in B(0,3\rho)$,
\[
         P^x[T_{A}<T_{B(0,M\rho)^c}] \ge \delta \g.
\]
Observing that  $g(\rho)\ge g$ on $S(\rho)$  the proof is finished.
\end{proof}

Now let us  fix  a~locally finite subset $Z$ of $\reald\setminus \{0\}$ and $r_z>4R_0$, $z\in Z$,  such that the 
 balls $  B(z,r_z)$ are pairwise disjoint and satisfy the separation condition
(\ref{separation}). 
Let $A$ denote the union of these balls. 
For a~proof of  Theorem \ref{main}  we show the following.

\begin{proposition}\label{only-to-show}
If  $A$ is avoidable,  then $\sum_{z\in Z} {g(|z|)} /{g(r_z)}<\infty$.
\end{proposition} 

\begin{proof} So let us suppose that $A$ is avoidable. To prove that $\sum_{z\in Z} {g(|z|)} /{g(r_z)}<\infty$
we may  assume that $|z|>8R_0$, for every $z\in Z$  (we simply omit finitely many points from $Z$). 
Further, we may assume   that the balls $B(z, 4r_z)$  are pairwise
disjoint. Indeed,   since 
$g(r)\le g(r/4) \le c^2 g(r)$, $r>R_0$,
 a~replacement of~$r_z$ by~$r_z/4$ does neither affect~(\ref{separation}) nor the convergence of $\sum_{z\in Z} {g(|z|)} /{g(r_z)}$,
 and the new, smaller union is, of course, avoidable.  Moreover, similarly as at  the beginning of the proof of \cite[Theorem 1.1]{mimica-vondracek}),
we may assume without loss of generality that 
\begin{equation}\label{rz}
                          r_z\le |z|/8, \qquad\mbox{ for every }z\in Z.
\end{equation} 
Indeed, replacing $r_z$ by  
$
r_z':=\min\{r_z,|z|/8\}
$
our assumptions are preserved as well.
Suppose we have shown that   $\sum_{z\in Z} {g(|z|)} /{g(r_z')}<\infty$.
 Since $g(|z|)/g(|z|/8)\ge  c^{-3}$,  
  we see that the set $Z'$ of all $z\in Z$ such that $r_z'=|z|/8$ is finite, and hence certainly 
$\sum_{z\in Z'} g(|z|)/g(r_z)<\infty$. So  we may assume without loss of generality 
that $r_z'=r_z$, for all $z\in Z$, that is,  (\ref{rz}) holds.

 By (\ref{separation}),   we may choose $0<\ve< 1/4$ such that,  for  $z,z'\in Z$,  $ z\ne z'$,
\begin{equation}\label{strong-sep}
  |z-z'|  \ge 8c^{1/d}\ve |z| \bigl({g(|z|)}/{g(r_z)}\bigr)^{1/d}. 
\end{equation} 
As in Lemma \ref{shell-1}, we define
\[ 
           C:=  1+ (4/\ve)^dc^3, \qquad \delta:=(2C  c^4)\inv.
\] 
By  Lemma \ref{UD},  there exists  $M:=3^m$, $m\in\nat$, such that
\[ 
g((M-3)\rho)\le \delta g(\rho),\qquad\mbox{ for every }\rho> R_1.
\] 
Moreover, let us define 
\[
R:=1+\max\{R_0,R_1\}.
\]

By Proposition  \ref{so-simple},  there is  a~point $x_0$  in $\reald$
such that 
\begin{equation}\label{start}
           P^{x_0}[T_A<\infty] = R_1^A(x_0)< \delta/2. 
\end{equation} 
Deleting finitely many points from $Z$,                                                  
we obtain $Z\cap B(0,2|x_0|+R)=\emptyset$. Then, for every $z\in Z$,
\begin{equation}\label{Z-modified}
|z|/2 \le |z-x_0|\le 2|z|, \qquad c\inv g(|z|)\le g(|z-x_0|)\le c g(|z|).
\end{equation} 
Hence, by (\ref{rz})  and (\ref{strong-sep}), $r_z<|z-x_0|/4$ and, for   $ z,z'\in Z$, $ z\ne z'$,  
\begin{equation}\label{translation}
|z-z'|\ge 4\ve |z-x_0| \bigl( {g(|z-x_0|)}/{g(r_z)}\bigr)^d. 
\end{equation} 
By translation invariance, we may therefore assume without loss of generality that $x_0=0$, $Z\cap B(0,R)=\emptyset$,
 and  (\ref{translation}) holds instead of~(\ref{strong-sep}).

For every $0\le j<m$, let 
\[
 Z_j:=\bigcup\nolimits_{n=0}^\infty Z\cap  S(3^{nm+j}R). 
\]
Then  $Z$ is the union of $Z_0,Z_1,\dots, Z_{m-1}$. Therefore it suffices to show that
\begin{equation}\label{Z-show}
                \sum\nolimits_{z\in Z_j} g(|z|)/g(r_z) <\infty, \qquad \mbox{ for every } 0\le j<m.
\end{equation} 
So let us fix $0\le j< m$.  For the moment, we also fix  $n\in \{0,1,2,\dots\}$ and define
$ \rho:=3^{nm+j}R$,
\[
      S:=T_{S(\rho)} , \quad \tau:=T_{B(0,\rho)^c}, \quad \tau':=T_{B(0,M\rho)^c}, \quad T:=\min\{T_A,\tau'\}.
\]
By  Lemma \ref{shell-1}, 
\begin{equation*} 
 P^y[T_A<\tau']\ge
\delta \sum\nolimits_{z\in Z\cap S(\rho)} g(|z|)/g(r_z), \quad\mbox{ for every }  y\in S(\rho). 
\end{equation*} 
By   Proposition \ref{hame}, $P^0[S<\tau'] \ge \delta$, and hence, by (\ref{start}), 
\begin{equation*} 
    P^0[S<T] \ge P^0[S<\tau']- P^0[T_A<\infty] >\delta/2. 
\end{equation*} 
Clearly, $S+T_A\circ\theta_{S}= T_A$ and $S+\tau'\circ\theta_{S}= \tau'$ on $[S<T] $. Hence
\[
 [S<T_A<\tau']=[S< T, T_A<\tau']=[S< T]\cap \theta_{S}\inv ([T_A<\tau']).
\]
%Further, $X_{S}\in S(\rho)$ on $[S<\infty]$. Therefore, by   the strong Markov property and~(\ref{partial}), 
Since $X_{S}\in S(\rho)$ on $[S<\infty]$,   the strong Markov property yields  that 
\[
       P^0[S<T_A<\tau']=\int_{[S<  T]} P^{X_{S}}[T_A<\tau']\,dP^0 \ge
(\delta^2/2)  \sum\nolimits_{z\in Z\cap S(\rho)  }  \frac{g(|z|)}{g(r_z)}.
\]
Of course, $\tau\le S$. Hence the sets $[S<T_A<\tau']$, obtained for different~$n$, are pairwise disjoint
subsets of $[T_A<\infty]$ (recall that $M=3^m$). Thus, by~(\ref{start}),  
\[
\sum\nolimits_{z\in Z_j} g(|z|)/g(r_z)\le (2/\delta^2)   P^0[T_A<\infty] \le 1/\delta.
\]
    \end{proof} 

 Let us note that the preceding proof could also be presented in a~purely analytic way
using iterated balayage of measures.

\section{Proof of Corollary \ref{main-corollary}}

Again we suppose that the assumptions from the beginning of Section 1 are satisfied.
Let $Z$ be a~countable set in $\reald$ and $r_z>0$, $z\in Z$, such that the balls $B(z,r_z)$ 
are pairwise disjoint and regularly located. So there exist $\ve,R\in (0,\infty)$ such that the points
in $Z$ have a~mutual distance which is at least $\ve$ and every open ball of radius $R$
contains some point of $Z$. Moreover, $r_z=\phi(|z|)$, where the function $\phi$ is decreasing.
If (LD) does not hold with $R_0=0$, we assume that $\kappa:=\inf_{x\in \reald}  \phi(x)>0$.
By Lemma~\ref{LD}, we then know that (LD) holds, if we define $R_0:=\kappa/8$.
 By Lemma~\ref{UD},   (UD) holds with, say,  $R_1:=R_0+1$.     

Of course, we may assume   that $R\ge 1+\phi(1)$. 
We already know (see Remark \ref{remark-main},3 and Proposition  \ref{known})      
that it suffices to show that the union~$A$ of all $B(z,r_z)$, $z\in Z$, is unavoidable provided 
\begin{equation}\label{nec}
\sum\nolimits_{z\in Z} g(|z|)/g(r_z)=\infty.
\end{equation} 

So let us suppose that (\ref{nec}) holds. Moreover, let us assume for the moment that  
\begin{equation}\label{sufficient}
\limsup\nolimits_{\rho\to\infty} \rho^dg(\rho)/g(\phi(\rho))<\infty.
\end{equation} 
Then 
$\b:=  \inf\nolimits_{z\in Z} g(r_z)(|z|^d g(|z|))\inv >0$. 
Since $|z-z'|\ge \ve>0$, whenever  $z\ne z'$,  this implies that
\[
           \inf\nolimits_{z,z'\in Z, z\ne z'}  \frac  {|z-z'|^d}{|z|^d}        \, \frac {g(r_z)} {g(|z|)} \ge \ve^d\b.
\]
Hence the balls $B(z,r_z)$, $z\in Z$, satisfy the separation condition (\ref{separation}), and   $A$~is 
unavoidable,  by Theorem \ref{main}. Thus already the following lemma would finish        
the proof of Corollary \ref{main-corollary}.

\begin{lemma}\label{sufficient-rho}
If $\limsup_{\rho\to\infty} \rho^dg(\rho)/g(\phi(\rho))=\infty$, the set $A$ is unavoidable.
\end{lemma}

For L\'evy processes considered in \cite{mimica-vondracek}, this is \cite[Lemma 4.2] {mimica-vondracek}. 
However, its proof (by contradiction) is almost as involved as the proof of \cite[Theorem~1.1]{mimica-vondracek}.

By Proposition  \ref{so-simple},  we only have to show that $\inf_{x\in \reald} R_1^A(x)>0$. Hence
a~second application of Lemma \ref{comparison},    which is  yet 
another variation of the arguments for quasi-additivity of capacities in~\cite{aikawa-borichev},
will allow us even to prove  the following. 

\begin{proposition}\label{MV-L4.2}
Suppose that $\limsup_{\rho\to\infty} \rho^dg(\rho)/g(\phi(\rho)) 
>\eta>0$.       
 Then the union $A$ of all $B(z,r_z)$, $z\in Z$, is unavoidable.
\end{proposition} 

\begin{proof}             
We define
\[
      a:=  (2c)\inv  (18R)^{-d}  , \qquad b:= c^2 R^{-d},
\]
and  fix $x\in\reald$. 
There exists $\rho> 9R +2|x|+4 R_1$ such that 
\begin{equation}\label{rho-choice}
\g:=\rho^d g(\rho)/g(\phi(\rho))> \eta. 
\end{equation} 
Let
\[
      r:=\phi(\rho),\qquad     B:=B(0,\rho) \und  S:=\ov B(0,\rho/2)\setminus B(\rho/4).
\]
There exist finitely many points $y_1,\dots,y_m\in S$ 
 such that    $B(y_1,3R),\dots, B(y_m,3R)$ are pairwise disjoint and
$S$ is covered by the balls $B(y_1,9R), \dots, B(y_m,9R)$.
Obviously, $m\ge(1/2) (\rho/18R)^d$. There exist points $z_j\in Z\cap B(y_j,R)$, $1\le j\le m$. Then,
for all $i,j\in \{1,\dots, m\}$ with $i\ne j$, $|z_i-z_j|\ge |y_i-y_j|-2R\ge 4R$, and hence
\begin{equation}\label{distance} 
B(z_i,R)\cap B(z_j,3R)=\emptyset.
\end{equation} 
Let  $1\le j\le m$. Clearly, $\rho\ge \rho/2+R\ge |z_j| \ge \rho/4-R\ge R\ge 1$, and hence 
\begin{equation}\label{rzR}
                r=\phi(\rho)\le \phi(|z_j|)=r_{z_j} \le \phi(1)\le R.
\end{equation} 
Moreover, $r+|x-z_j| \le R+|x|+ \rho/2+R\le  \rho$,
and hence $g(|x-z_j|+r)\ge  g(\rho)$. So, by translation invariance and    Proposition~\ref{cap},
\begin{equation}\label{Rzj}
                 R_1^{B(z_j,r)}(x)\ge c\inv g(|x-z_j|+r)/g(r) \ge   c\inv g(\rho)/g(r).
\end{equation} 
Let
\[
         A_x:=B(z_1,r)\cup\dots\cup B(z_m,r) \und  p:=   \sum\nolimits_{j=1}^m R_1^{B(z_j,r)}.
\]
Then $A_x\subset A$, by (\ref{rzR}).  So, by (\ref{Rzj}) and our definitions of $\g$, $a$, and $r$,       
\begin{equation}\label{est-p}
R_1^{A_x}\le R_1^A  \und p(x) \ge m  c\inv g(\rho)/g(r ) \ge   a\g .
\end{equation} 

Now let $\mu_0$ denote the equilibrium measure for $B(0,r)$, %that is, 
$G\mu_0=R_1^{B(0,r)}$.  By Proposition \ref{cap}, $\|\mu_0\|\le c g(r)\inv$. We  define
\[
          \nu_j:=\|\mu_0\| \lam_{B(z_j,R)}=\|\mu_0\| (\rho/R)^d  1_{B(z_j,R)} \lambda_B,  \qquad 1\le j\le m,
\]
and $\nu:=\sum_{j=1}^m \nu_j$.
Since   $B(z_1,R), \dots, B(z_j, R)$ are pairwise disjoint subsets of~$B$ and 
$G\lambda_B\le  c g(\rho)$,  we see that
\[ 
       G\nu\le   \|\mu_0\| (\rho/R)^d G\lambda_B \le c^2 R^{-d}g(r)\inv \rho^d g(\rho) = b\g.
\] 
 
For every $1\le j\le m$,  $R_1^{B(z_j,r)}=G\mu_j$, where $\mu_j$ is obtained from~$\mu_0$ translating by $z_j$. 
  Let $\mu:=\sum_{j=1}^m \mu_j$.
By (\ref{distance}), (\ref{rzR}),  and Lemma \ref{comparison}, 
\[
                p=      G\mu\le 1+c G\nu\le 1+c b\g.
\]
Since $\mu$ is supported by $\ov A_x$ and $p$ is continuous, we get that 
% \in \C(\reald)$, 
\[
             R_1^{A_x}=R_1^{\ov A_x}\ge  (1+c b\g)\inv p, 
\]
by  the minimum principle~\cite[III.6.6]{BH}. In particular, 
\[
           R_1^A(x)\ge  R_1^{A_x}(x)\ge \frac{a\g}{1+cb\g}=\frac a{\g\inv +cb} > \frac a {\eta \inv +c b}\,
\]
by (\ref{est-p}) and (\ref{rho-choice}).  Thus  $A$ is unavoidable, by Proposition  \ref{so-simple}. 
\end{proof}

\bibliographystyle{ws-rv-van}

\end{document}